\documentclass[12pt]{article}

\usepackage{amsthm,amsmath,amssymb}

\usepackage{graphicx}

\usepackage[colorlinks=true,citecolor=black,linkcolor=black,urlcolor=blue]{hyperref}

\theoremstyle{plain}
\newtheorem{theorem}{Theorem}
\newtheorem{lemma}[theorem]{Lemma}

\newtheorem{claim}[theorem]{Claim}

\theoremstyle{definition}
\newtheorem{definition}[theorem]{Definition}

\theoremstyle{remark}

\title{Pancyclicity when each cycle contains $k$ chords}

\author{Vladislav Taranchuk\thanks{Department of Statistics and Mathematics, California State University Sacramento, Sacramento, CA, 95819, vtaranchuk@csus.edu. This research was funded by SURE Grant by the College of Natural Sciences and Mathematics from the California State University Sacramento.}}

\begin{document}

\maketitle 

\begin{abstract}
For integers $n \geq k \geq 2$, let $c(n,k)$ be the minimum number of chords that must be added to a cycle of length $n$ so that 
the resulting graph has the property that for every $l \in \{ k , k + 1 , \dots , n \}$, there is a cycle of length $l$ that contains 
exactly $k$ of the added chords.  
Affif Chaouche, Rutherford, and Whitty introduced the function $c(n,k)$. They showed that for every integer $k \geq 2$, 
$c(n , k ) \geq \Omega_k ( n^{1/k} )$ and they asked if $n^{1/k}$ gives the correct order of magnitude of $c(n, k)$ for $k \geq 2$. Our main theorem answers this question as we prove that for every integer $k \geq 2$, and for sufficiently large $n$, $c(n , k) \leq k \lceil n^{1/k} \rceil + k^2$.
This upper bound, together with the lower bound of Affif Chaouche et.\ al., shows that the order of magnitude of 
$c(n,k)$ is $n^{1/k}$.
\end{abstract}

\section{Introduction}

An $n$-vertex graph is said to be \textbf{pancyclic} if it contains a cycle of length $l$ for each $l \in \{ 3, 4, 5, \dots , n \}$. There is a large amount of research on pancyclic graphs including many papers on conditions which imply pancyclicty, as well as investigations into properties of pancyclic graphs. For more on pancyclic graphs, we refer the reader to the recent book of George, Khodkar, and Wallis \cite{book}. Our focus will be on an extremal function that has its roots in a function of Bondy. Let $m(n)$ be the minimum number of edges in a pancyclic graph on $n$ vertices. Bondy \cite{b} introduced the function $m(n)$ and stated, without proof, that
\begin{equation}\label{Bon}
n-1 + \log_2(n-1) \leq m(n) \leq n + \log_2 n + H(n) + O(1)
\end{equation}
where $H(n)$ is the smallest integer such that $(\log_2)^{H(n)}(n) < 2$. To our knowledge, the bounds given in (\ref{Bon}) have not been improved. Nevertheless, Bondy's problem is quite natural and has inspired several new extremal functions. Let us take a moment to introduce one such example before we define the extremal function that we will focus on.

Broersma \cite{br} asked for the minimum number of edges in a vertex pancyclic graph on $n$ vertices. Recall an $n$-vertex graph is \textbf{vertex pancyclic} if every vertex lies in a cycle of length $l$ for every $l \in \{ 3, 4, 5, \dots , n \}$. Broersma proved that 
\begin{equation}\label{bro}
\frac{3}{2}n < vp(n) \leq \lfloor \frac{5}{3}n \rfloor
\end{equation}
for all $n \geq 7$ where $vp(n)$ is the minimum number of edges in a vertex pancyclic graph with $n$ vertices. This shows that $vp(n)$ is linear in $n$, but an asymptotic formula for $vp(n)$ is not known. Improving either (\ref{Bon}) or (\ref{bro}) would be quite interesting.

A consequence of (\ref{Bon}) is that if $m'(n)$ is the minimum number of chords that must be added to $C_n$ to obtain a pancyclic graph, then
\[
m'(n) = \log_2 n + o(\log_2 n).
\] 
Similarly, (\ref{bro}) implies that the minimum number of chords that must be added to $C_n$ to obtain a vertex pancyclic graph is $\Theta(n)$. Motivated by these results, Affif Chaouche, Rutherford, and Whitty \cite{crw} introduced the following extremal function.
For integers $n \geq 6 $ and $k \geq 2$, let $c(n,k)$ be the minimum number of chords that must be added to a cycle 
of length $n$ so that the resulting graph has the property that for 
every integer $l \in \{k , k + 1 , \dots , n \}$, there is a cycle of length $l$ that contains exactly $k$ of the added 
chords. By (\ref{Bon}), we immediately get $c(n, k) \geq \log_2(n-1)$ for all $k \geq 2$. The following result of Affif Chaouche et.\ al.\ (see corollary 8 in \cite{crw}) improves this lower bound and shows that requiring each cycle to contain exactly $k$ of the added chords has a rather dramatic effect on the amount of chords that must be added.

\begin{theorem}[Affif Chaouche, Rutherford, Whitty \cite{crw}]\label{thm crw}
Let $k \geq 1$ be an integer.  For any integer $n \geq 6$
\[
c(n, k ) \geq \Omega( n^{1/k}).
\]
\end{theorem}

In \cite{crw}, it is suggested that $c(n, k)$ lies somewhere between $m'(n) = \Theta (\log n)$ and $vp (n) = \Theta(n)$. It was left as an unsolved problem to prove a non-trivial upper bound on $c(n, k)$. Our main result solves this problem completely, and together with Theorem \ref{thm crw}, shows that the order of magnitude of $c(n, k)$ is $n^{1/k}$. This result comes in the form of a construction that uses the representation of cycle lengths in a specific base dependent on $n$ and $k$. 

\begin{theorem}\label{Main Theorem}
Let $k \geq 2$ be an integer. If $ n\geq (k+2)^k,$ then
$$c(n, k) \leq k \lceil n^{1/k} \rceil + k^2. $$
\end{theorem}

Using Stirling's approximation, a close look at the proof of Theorem \ref{thm crw}, shows that the lower bound $\Omega( n^{1/k})$ is asymptotic to $\frac{k}{e}n^{1/k}$ as $k$ goes to infinity. Therefore, our Theorem \ref{Main Theorem} is best possible up to a constant factor. In fact, we believe that for $k \geq 2$, $c(n, k) = kn^{1/k} + o(n^{1/k})$ and Theorem \ref{Main Theorem} is asymptotically best possible. 

In the next section we prove Theorem \ref{Main Theorem}. Some concluding remarks are made in Section 3. Most of the notation that we use follows that of West \cite{west}.  
In particular, $C_n$ always denotes a cycle of length $n$ with vertex set $\{1,2, \dots , n \}$ whose edges are $\{ i , i + 1\}$, $i \in \{1,2, \dots , n \}$ together with $\{ n , 1 \}$.   

\section{Proof of Theorem \ref{Main Theorem}}

Let $k \geq 2$ be an integer and let $ n\geq (k+2)^k$. In order to prove that 
\[
c(n, k) \leq k \lceil n^{1/k} \rceil +k^2,
\]
we must show how to add 
at most $k \lceil n^{1/k} \rceil + k^2$ chords to $C_n$ so that in the resulting graph, 
for each $l \in \{ k, k+1,\dots , n \}$, there is a cycle of length $l$ that contains exactly $k$ chord edges.
The construction is best described in several steps.  

\noindent
\underline{Step 1}: Define the graphs which will serve as the building blocks for our construction.
These are the graphs $G_b (i,e)$, defined in Definition 1, and the base chords of 
$G_b (i,e)$ will be chord edges in the final construction.

\medskip
\noindent
\underline{Step 2}: Prove that each $G_b (i,e)$ has paths of certain lengths between two specific vertices (see Lemma \ref{Base graph}).

\medskip
\noindent
\underline{Step 3}: Form the graph $H_b (k)$ (see Definition 2) which is 
obtained by taking the union of a certain collection of $G_b (i,e)$'s.  

\medskip
\noindent
\underline{Step 4}: Show that $H_b (k)$ contains paths of certain 
lengths between two specific vertices (see Lemma \ref{long path}).  Each of these paths 
will contain exactly $k-1$ edges that will end up being chords in the final construction.  Furthermore, 
these paths are obtained by taking the union of paths whose existence is established in Step 2.  

\medskip
\noindent
\underline{Step 5}: Add a few more edges to $H_b (k)$ (see Definition 3) to obtain a 
graph that has a cycle of length $l$ for 
all $l \in \{ k, k+1 , \dots , n - k \}$ that contains exactly $k$ chords 

\medskip
\noindent
\underline{Step 6}: Finish the construction by adding a small number of chords to account 
for the needed cycles of length $n - k +1 , n-k + 2, \dots , n $.  

\bigskip
Having given a brief outline, let us proceed to the details.   

\begin{definition}[\textbf{1}]
Let $b \geq 3$, $i \geq 1,$ and $e \geq 0$ be integers. Let $G_b(i, e)$ be the graph with vertex set 
$$\{i, i+1, i+2,\dots , i+2+b^{e+1}-b^e \}$$ and edge set 
$$\{ \{j, j+1 \} :i \leq j \leq i+1+b^{e+1}-b^e \} \cup \{ \{ i, i+2+jb^e \}: 0 \leq j \leq b-1 \} .$$  
The vertex $i$ is called the \textbf{base point} of $G_b(i,e)$.
Edges in the set 
$$\{ \{ i, i+2+jb^e \}: 0 \leq j \leq b-1 \}$$
are called \textbf{base chords}.
Edges of the form $\{j , j + 1 \}$ are called \textbf{outer edges}.
Figure 1, given below, shows three different $G_b (i,e)$'s.    
\end{definition}

Observe that in $G_b(i, e)$ there are exactly
$ b^{e+1}-b^e+2 $ outer edges and exactly $b$ base chords. The vertices of the form $b^x +2x$ play an important role in our construction, so we define 
$$
q(x) = b^x + 2x.
$$ 
 with $x \in \{ 0, 1, \dots, k \}$.
\begin{lemma}\label{Base graph}
For each $c_e \in \{ 0, 1, \dots , b-1 \}$, the graph $G_b(q(e), e)$ contains a path $P$ of length $1+c_eb^e$ from 
the vertex $q(e)$ to the vertex $q(e+1)$ such that $P$  contains exactly one base chord.
\end{lemma}
\begin{proof}[Proof of Lemma \ref{Base graph}]
Let $c_e \in \{ 0, 1, 2, \dots, b-1 \}$ and define $x = q(e+1) - c_e b^e$.  Consider the path $P$ whose first edge
is $\{q(e), x \} $, and whose remaining edges are  
$$ \{ \{ x+j, x+j+1 \} :0 \leq j \leq c_eb^e -1 \}.$$
The path $P$ has $1+c_eb^e$ edges.  
The only base chord in the path is $\{ q(e), x \}$. 
To show that this edge is in fact a base chord, note that
$$x = q(e+1) - c_e b^e = b^{e+1}+2(e+1)-c_eb^e = (b^e+2e )+2+(b-c_e-1)b^e$$ and $b-c_e-1 \in \{ 0, 1, \dots , b-1 \}$ since 
$c_e \in \{ 0, 1, \dots , b-1 \}$.
\end{proof}

\bigskip

Note that in the graph $G_b(q(e), e)$, $q(e)$ and $q(e+1)$ are the first and last vertices in this graph respectively. The graphs $G_b (q(e) , e)$ form the building blocks of our construction and they will be put together using the graph union operation. We slightly alter the definition of a graph union for our purposes due to the fact that the graphs $G_b (q(e) , e)$ and $G_b (q(e+1) , e+1)$ share exactly one vertex, namely $q(e+1)$.
If $G_1, \dots , G_r$ is a collection of graphs, then the graph $G_1 \cup G_2 \cup \dots \cup G_r$ is the graph with vertex set 
$$V(G_1) \cup V(G_2) \cup \dots \cup V(G_r)$$ 
and edge set 
$$E(G_1) \cup E(G_2) \cup \dots \cup E(G_r).$$

\begin{definition}[\textbf{2}]
For integers $k \geq 2$ and $b \geq 3$, let $H_b(k)$ be the graph $$H_b(k) = \bigcup_{e=0}^{k-1} G_b(q(e), e).$$
\end{definition}

For example, the graph $H_4(4)$ is the union of the four graphs $G_4(1, 0)$, $G_4(6, 1)$, $G_4(20 , 2)$ and $G_4(70, 3)$. The first three of these graphs are shown 
in Figure 1.  In the figure, all of the base chords are shown, but not all of the vertices and outer edges of $G_4 (6,1)$ and $G_4 (20 , 2)$ are shown.  

\begin{center}
\begin{picture}(420,100)
\put(0,30){\circle*{3}}
\put(30,30){\circle*{3}}
\put(60,30){\circle*{3}}
\put(90,30){\circle*{3}}
\put(120,30){\circle*{3}}
\put(60,90){\circle*{3}}
\put(0,30){\line(1,1){60}}
\put(30,30){\line(1,2){30}}
\put(60,30){\line(0,1){60}}
\put(90,30){\line(-1,2){30}}
\put(120,30){\line(-1,1){60}}
\put(0,30){\line(1,0){120}}
\put(57,95){$1$}
\put(-5,17){$2$}
\put(25,17){$3$}
\put(55,17){$4$}
\put(85,17){$5$}
\put(115,17){$6$}
\put(40,0){$G_4(1,0)$}

\put(150,30){\circle*{3}}
\put(180,30){\circle*{3}}
\put(210,30){\circle*{3}}
\put(240,30){\circle*{3}}
\put(270,30){\circle*{3}}
\put(210,90){\circle*{3}}
\put(150,30){\line(1,1){60}}
\put(180,30){\line(1,2){30}}
\put(210,30){\line(0,1){60}}
\put(240,30){\line(-1,2){30}}
\put(270,30){\line(-1,1){60}}

\put(207,95){$6$}
\put(145,17){$7$}
\put(175,17){$8$}
\put(205,17){$12$}
\put(235,17){$16$}
\put(265,17){$20$}
\put(190,0){$G_4(6,1)$}

\put(150,30){\line(1,0){30}}
\put(180,30){\line(1,0){10}}
\put(210,30){\line(1,0){10}}
\put(210,30){\line(-1,0){10}}
\put(240,30){\line(1,0){10}}
\put(240,30){\line(-1,0){10}}
\put(270,30){\line(-1,0){10}}

\put(300,30){\circle*{3}}
\put(330,30){\circle*{3}}
\put(360,30){\circle*{3}}
\put(390,30){\circle*{3}}
\put(420,30){\circle*{3}}
\put(360,90){\circle*{3}}
\put(300,30){\line(1,1){60}}
\put(330,30){\line(1,2){30}}
\put(360,30){\line(0,1){60}}
\put(390,30){\line(-1,2){30}}
\put(420,30){\line(-1,1){60}}

\put(357,95){$20$}
\put(295,17){$21$}
\put(325,17){$22$}
\put(355,17){$38$}
\put(385,17){$54$}
\put(415,17){$70$}
\put(340,0){$G_4(20,2)$}

\put(300,30){\line(1,0){30}}
\put(330,30){\line(1,0){10}}
\put(360,30){\line(1,0){10}}
\put(360,30){\line(-1,0){10}}
\put(390,30){\line(1,0){10}}
\put(390,30){\line(-1,0){10}}
\put(420,30){\line(-1,0){10}}

\end{picture}

\vspace{.75em}

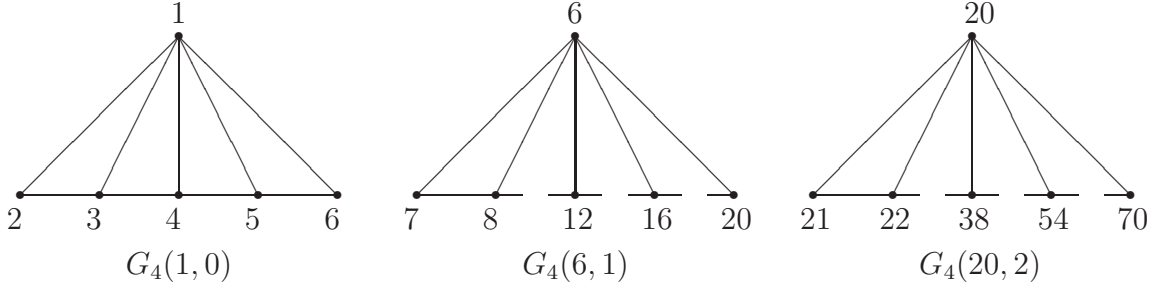
Figure 1: The graphs $G_4(1,0)$, $G_4 (6,1)$, and $G_4(20 , 2)$.  
\end{center}

Recall that $q(x) = b^x +2x$. For $0 \leq e_1 < e_2 \leq k-1 $, the intersection 
$$V(G_b(q(e_1), e_1)) \cap V(G_b(q(e_2), e_2))$$
is empty unless $e_2 =e_1+1$ in which case the intersection is precisely 
$q(e_1 + 1)$. 
This implies that the graphs 
\[
G_b(q(0), 0), G_b(q(1), 1), G_b(q(2) , 2) ,  \dots , G_b(q(k-1) , k -1)
\]
are edge disjoint.  Therefore, $H_b(k)$ has $b^{k}+2k$ vertices and $b^k +2k -1 +bk$ edges.  An edge of $H_b (k)$ that 
is a base chord of some $G_b ( q(e) , e )$ is called a \textbf{chord edge}.  As 
each $G_b (q(e) , e)$ has $b$ base chords, the graph $H_b (k)$ has $bk$ chord edges.  The edges of 
$H_b (k)$ that are not chord edges form a path from the vertex 1 to the vertex $q(k)$.  
Referring to Figure 1, the chord edges of $H_4 (3)$ are $\{ 1 , i \}$ for $i \in \{3,4,5,6 \}$, $\{6,i \}$ for 
$i \in \{ 8, 12 , 16 , 20 \}$, and $\{ 20 , i \}$ for $i \in \{ 22, 38 , 54 , 70 \}$.  The remaining edges of $H_4 (4)$ form a 
path from the vertex 1 to the vertex $q(3) = 4^3 + 2 \cdot 3 = 70$.  



The following lemma is key to our construction.  Before proving the lemma, let us give a quick example
using the graph $H_4 (4)$.  Suppose we want to find a path of length 52 from 1 to 70 that uses exactly 3 chord edges.  
We first write
$$
52-3 = 49 = 3\cdot 4^2 + 1\cdot 4^1 + 1\cdot 4^0.
$$
Now consider the path from 1 to 70 that uses the base chords $\{1, 5  \}, \{ 6, 16 \}, \{ 20, 22 \}$. We can see that this path contains exactly 49 outer edges and 3 chord edges, and so has length of 52. Notice that we can write these base chords the following way: $\{ 1, 5 \} = \{ q(0), q(1) - 1 \cdot 4^0 \} $, $\{6, 16 \} = \{ q(1), q(2) - 1\cdot 4^1 \}$, and $\{ 20, 22 \} = \{ q(2), q(3) - 3\cdot 4^2 \}$. So then it is the coefficients of $4^e$ that determine which base chord to take.


\begin{lemma}\label{long path}

For each $l \in \{ k-1, k, \dots, k + b^{k-1} - 2 \}$, there exists a path of length $l$ in the graph $H_b(k)$ from the vertex 1 to the vertex $q(k-1)$ that contains exactly $k-1$ chord edges.

\end{lemma}

\begin{proof}[Proof of Lemma \ref{long path}]
Let $e \in \{0,1, \dots , k -2 \}$.  
By Lemma \ref{Base graph}, there is a path $P_e$ of length $1+c_eb^e$ in the unique copy of $G_b(q(e), e)$ in $H_b(k)$
where the first vertex of $P_e$ is $q(e)$, the last vertex is $q(e+1)$, and $P_e$ contains exactly one chord edge. 
Therefore, the union $$P_0 \cup P_1 \cup \dots \cup P_{k-2}$$ is a path of length $(k-1) + \sum_{e=0}^{k-2}c_eb^e$ from the vertex $1$ to the vertex $q(k-1)$ that contains exactly $k-1$ chord edges. 

It is worth commenting that we are essentially using representations of integers in base $b$.  
Given $l \in \{ k , k + 1, \dots , k + b^{k -1} - 2 \}$, there are unique 
integers $c_0 , c_1 , \dots , c_{k-2} \in \{0,1, \dots , b-1 \}$ such that 
\[
l - (k-1) = c_0 b^0 + c_1 b^1  + \dots + c_{k-2}b^{k-2}.
\]
In other words, we have written $l - (k-1)$ in base $b$. This last equation implies that 
\[
l = ( 1 + c_0 b^0 ) + (1 + c_1 b^1 ) + \dots + (1  + c_{k-2}b^{k-2} )
\]
and the use of Lemma \ref{Base graph} accounts for each term in this sum.  
\end{proof}
We will require another definition in order to continue the description of the construction. This definition is dependent on the following claim.

\begin{claim}\label{Claim 1}
Let $b$ be a positive integer.  
If $b = \lceil n^{1/k} \rceil$, then we have $n > b^{k-1} + 2k$ and $(b-1)^k < n \leq b^k$.
\end{claim}

\begin{proof}[Proof of Claim \ref{Claim 1}]
Since $b = \lceil n^{1/k} \rceil$, we immediately have $(b-1)^k < n \leq b^k$. 
Using $n > (b-1)^k$ we confirm that $n > b^{k-1} + 2k$ by showing that the inequality
\\ $(b-1)^k > b^{k-1} + 2k$ holds. Now $k \geq 2$ and $b = \lceil n^{1/k} \rceil \geq \lceil ((k + 2)^k)^{1/k} \rceil = k+2$, and when $b>k$,
it is known that $(b-1)^k$ grows more quickly than $b^{k-1}$, so it is sufficient to observe that the inequality holds when $k=2$ and $b=k+2$.
\end{proof}    

\begin{definition}[\textbf{3}]
Let $k \geq 2$ and $n \geq (k+2)^k$ be integers.  
Consider $b = \lceil n^{1/k} \rceil$. Add the edges 
\[
\{ 1 , q(k-1) \} ~ \mbox{and} ~ \{1 , n \}
\]
to $H_b(k)$. We then define $F_k(n)$ to be this altered $H_b(k)$ graph induced by the vertices $\{1,2, \dots , n \}$.   

Note that when $k=2$, the edge $\{1 , q(k-1) \}$ is already in $H_b(k-1)$. So when $k=2$ we do not add this chord again. The edge $\{1 , q(k-1) \}$ is called a \textbf{chord edge} and all of the chord edges in $H_b (k)$ are also called \textbf{chord edges} in $F_k (n)$.   
\end{definition}

Observe that the number of chord edges of $F_k (n)$ is at most $k \lceil n^{1/k} \rceil +1$ and 
that $F_k (n)$ contains a cycle of length $n$ whose edges are 
\[
\{1 , n \} \cup \{ \{ i , i + 1 \} : 1 \leq i \leq n - 1 \}.
\]

The number of total chord edges is what we are interested in counting to give an upper bound for $c(n, k)$. The next lemma will show that for all 
\[
l \in \{k , k + 1 , \dots , n - k \},
\]
we can find a cycle of length $l$ in 
$F_k (n)$ that contains exactly $k$ chords.  Therefore, the graph $F_k (n)$ contains almost all of the cycles that 
we need in order to complete the proof of Theorem \ref{Main Theorem}.    

The vertex $q(k-1) = b^{k-1} + 2 ( k - 1)$, 
which is the unique vertex in both 
$G_b ( q(k-2) , k-2)$ and 
$G_b ( q(k-1) , k -1)$, will play a special role so we let 
\[
m = q(k-1) = b^{k - 1} + 2 ( k - 1).
\]  

\begin{lemma}\label{Lemma 1} 
Let $k \geq 2$ and $n \geq (k + 2)^k$.  For each 
$l \in \{k , k + 1 , \dots , n - k \}$, the graph $F_k(n)$ contains a cycle of length $l$ that passes through exactly $k$ chord edges.  
\end{lemma}

\begin{proof}[Proof of Lemma \ref{Lemma 1}]
We will prove Lemma \ref{Lemma 1} by establishing several claims.  First we give a quick outline of the proof.  


\begin{enumerate}

\item Using Lemma \ref{long path}, find paths from the vertex 1 to $q(k-1)$ that have exactly $k -1$ chord 
edges.
\item Find cycles whose last chord edge is $\{ m, 1\}$.  This 
will be Claim 7 and we will use the paths from part 1 of this lemma.  We will write $L_1$ for the 
cycle lengths we have found here.    
\item Find paths from the vertex $m$ to 1 that use exactly one chord edge in 
$G_b(q(k-1), k-1)$. This will be Claim 8.  
\item Find the cycles by combining the paths from parts 1 and 3 of this lemma.  This will be Claim 9 and we write 
$L_2$ for the cycle lengths we have found here.   
\item Show that parts 2 and 4 of this lemma combined cover all cycle lengths $l \in \{k , k + 1 , \dots , n - k \}$, i.e., 
\[
\{ k, k+1, \dots , n-k \} \subseteq L_1 \cup L_2.
\]
This will be Claim 10.

\end{enumerate}

\bigskip

By Definition 3, $F_k(n)$ contains exactly one complete copy of each of the graphs 
\[
G_b(q(0), 0), G_b(q(1), 1), G_b (q(2) , 2 ) , \dots , G_b(q(k-2), k-2).
\]

By Lemma \ref{long path}, for any $c_0, c_1, \dots , c_{k-2} \in \{ 0, 1, \dots , b-1 \}$, 
there is a path, which we will denote by $P(c_0, c_1, \dots , c_{k-2})$, that has the following properties.
\begin{enumerate}
\item The first vertex is 1 and the last vertex is $m$.
\item The length of $P(c_0, c_1, \dots , c_{k-2})$ is $k-1 + \sum_{e=0}^{k-2}c_eb^e$.
\item The edges of $P(c_0, c_1, \dots , c_{k-2})$ are edges in the graphs
\[
G_b(q(0), 0), G_b(q(1), 1), G_b(q(2) , 2),  \dots , G_b(q(k-2), k-2).
\]
\item The path $P(c_0, c_1, \dots , c_{k-2})$ contains exactly $k-1$ chord edges of $F_k (n)$.
\end{enumerate}

We will now find paths in $F_k(n)$ from the vertex $m$ to the vertex 1 that use edges from 
$G_b(q(k-1), k-1)$, and either $\{ n, 1 \}$ or $\{ m, 1\}$.  Each of these paths will also contain exactly one chord edge.

\begin{claim}\label{claim L1}
For every $l \in \{ k, k+1, \dots , k+ b^{k-1}-1 \}$ the graph $F_k(n)$ has a cycle of length $l$ that contains exactly $k$ chord edges.
\end{claim}

\begin{proof}[Proof of Claim \ref{claim L1}]
By definition of $F_k(n)$, the edge $\{ m, 1 \}$ is a chord edge. Therefore,
given $c_0, c_1, \dots , c_{k-2} \in \{ 0, 1, \dots , b-1 \}$, the union
$$P(c_0, c_1, \dots , c_{k-2}) \cup \{ m, 1 \}$$
is a cycle of length $k + \sum_{e=0}^{k-2}c_eb^e$ that contains exactly $k$ chord edges.
Claim \ref{claim L1} now follows from the fact that every integer in the set $\{ 0, 1, \dots , b^{k-1}-1 \}$ can be written in the form $\sum_{e=0}^{k-2}c_eb^e$ for some $c_0, c_1 , \dots , c_{k - 2} \in \{ 0, 1, \dots , b-1 \}$. 
\end{proof}
Let 
\[
L_1 = \left\{ k + \sum_{e  = 0}^{k -2} c_e b^e : c_e \in \{ 0 ,1, \dots , b - 1 \} \right\} =
\{ k, k+1, \dots , k+b^{k-1}-1 \}.
\]

Since $V(G_b(m, k-1)) = \{ m, m+1, m + 2, \dots , q(k) \}$, 
by Claim 5 the graph $F_k(n)$ contains a non-trivial induced subgraph of $G_b(m, k-1)$. 
Now we will define an integer $\alpha$ that counts the number of base chords in $G_b(m, k-1)$ 
that are chord edges in $F_k(n)$. Let $\alpha$ be the unique integer in the set $\{ 0, 1, \dots , b-2 \}$ that satisfies 
\begin{equation}\label{def of alpha}
 m + 2 + \alpha b^{k-1} \leq n <  m + 2 + (\alpha + 1)b^{k-1}.
\end{equation}
Such an $\alpha$ exists by Claim \ref{Claim 1}. The base chords in $G_b(m, k-1)$ that are chords in $F_k(n)$ are precisely those edges in the set 
$$ \{ \{ m, m + 2 + jb^{k-1} \} : j \in \{ 0, 1, \dots , \alpha \} \}. $$

\begin{claim}\label{most paths}
For any $j \in \{ 0, 1, \dots , \alpha \}$, there is a path $Q(j)$ in $F_k(n)$ with the following properties.
\begin{enumerate}
\item The first vertex is $m$ and the last vertex is 1.
\item The length of $Q(j)$ is $n + 2 - (j+1)b^{k-1} - 2k.$
\item All of the edges of $Q(j)$, with the exception of $\{ n, 1 \}$, are edges in the partial copy of
$G_b(m, k-1)$ in $F_k(n)$. 
\item The path $Q(j)$ contains exactly one chord edge.
\end{enumerate}
\end{claim}

\begin{proof}[Proof of Claim \ref{most paths}]
Let $j \in \{0, 1, \dots , \alpha \}$ and let $Q(j)$ be the path 
\[
m , m + 2 + j b^{k - 1} , m + 2 + jb^{k-1} + 1, m + 2 + jb^{k-1} + 2, \dots , n-1 , n, 1.
\]
Clearly the first property holds for $Q(j)$. The length of $Q(j)$ is 
$$ 1 + (n + 1) - (m + 2 + jb^{k-1}) = 
n+2 - (b^{k-1} + 2(k-1) + 2 + jb^{k-1}).$$
This shows that the second property holds.  The third property follows from the definition of the graphs $G_b(i, e)$. 
Finally, the fourth property follows from the fact that the first edge of $Q(j)$, which is $\{ m, m + 2 + jb^{k-1} \}$, is the only chord edge in $Q(j)$.
\end{proof}

\bigskip

Given $c_0, c_1, \dots , c_{k-2} \in \{ 0, 1 , \dots , b-1 \}$ and $j \in \{  0, 1, \dots , \alpha \}$, the union $$P(c_0, c_1, \dots , c_{k-2}) \cup Q(j)$$ is a cycle in $F_k (n)$ of length
$$ k-1 + \sum_{e=0}^{k-2}c_eb^e + n + 2 -(j+1)b^{k-1}-2k$$
that contains exactly $k$ chords. This expression can be rewritten as
$$n + 1 - k - (j+1)b^{k-1} + \sum_{e=0}^{k-2}c_eb^e.$$
Let $$L_2 = \{ n + 1 - k - (j+1)b^{k-1} + \sum_{e=0}^{k-2}c_eb^e : j \in \{  0, 1, \dots , \alpha \} , c_e \in \{ 0, 1 , \dots , b-1 \} \}.$$

\begin{claim}\label{L2 claim}
If $L_2$ is defined as above then

\begin{center}

\[
\left\{ n-k+1 - (\alpha+1)b^{k-1}, n-k+1 - (\alpha+1)b^{k-1}+1, n-k+1 - (\alpha+1)b^{k-1}+2, \dots , n-k \right\} \subseteq L_2 . 
\]
\end{center}
\end{claim}

\begin{proof}[Proof of Claim \ref{L2 claim}]
It is easy to check that the smallest integer in $L_2$ is $$n-k+1 - (\alpha+1)b^{k-1} $$
and the largest integer in $L_2$ is $$n-k+ b^{k-1} - b^{k-1} = n-k. $$
We will now show that $L_2$ contains every integer between $n-k+1 - (\alpha+1)b^{k-1}$ and $n-k$.

As the $c_e$'s range over $\{0,1, \dots , b - 1 \}$, the sum $\sum_{e = 0}^{k - 2} c_e b^e$ ranges over all integers in the set $\{0 , 1, \dots , b^{k-1}-1 \}$ as shown in Lemma 4.  Thus, for any fixed $j \in \{0,1, \dots , \alpha \}$, 
the set 

\[
I_j = \{ n-k+1 - ( j + 1)b^{ k - 1} + \sum_{e = 0}^{ k - 2} c_e b^e : c_e \in \{0 ,1, \dots , b-1 \} \}
\]
is the interval 
\[
\{ n-k+1 - ( j +1) b^{ k - 1} , n-k+1 - ( j + 1)b^{ k - 1} + 1 , \dots 
, n-k+1 - j b^{k-1} - 1 \}.
\]

Note that for $0 < j \leq \alpha$, the largest integer in $I_j$ is $n-k+1 - j b^{k - 1}-1$ 
and the smallest integer in $I_{j +1}$ is $n-k+1 - jb^{k-1}$.  These two integers are consecutive and so the union 
\[ 
I_0 \cup I_1 \cup \dots \cup I_{ \alpha }
\]
contains all integers from $n-k+1 - ( \alpha + 1) b^{k - 1}$ to $n-k+1 - 0 \cdot b^{k -1} - 1 = n - k$.  This completes the proof 
of Claim \ref{L2 claim}.       
\end{proof}

\begin{claim}\label{L1 and L2 claim}
We have 
\[
\{ k, k+1, \dots , n-k \} \subseteq L_1 \cup L_2.
\]
\end{claim}

\begin{proof}[Proof of Claim \ref{L1 and L2 claim}]
To prove Claim \ref{L1 and L2 claim}, it is enough to show that 
\begin{equation}\label{ineq on n 2}
n + 1 - k - ( \alpha + 1) b^{k -1} - 1 \leq k  + b^{k-1} - 1
\end{equation}
since the largest integer in $L_1$ is $k+b^{k - 1} - 1$, and the smallest integer in $L_2$ is 
$n +1 - k - ( \alpha + 1) b^{ k - 1}$.  
The inequality (\ref{ineq on n 2}) 
is equivalent to 
\begin{equation}\label{ineq on n}
n+1 \leq ( \alpha + 2) b^{k - 1} + 2k.
\end{equation}
Recalling the definition of $\alpha$ given in (\ref{def of alpha}), 
we have that 
$n < ( \alpha  +2 ) b^{ k - 1} + 2k$.  Now since $n$ and $( \alpha + 2) b^{k - 1}+2k$ are integers, this last inequality 
implies (\ref{ineq on n}).  
\end{proof}

\bigskip

Combining Claim \ref{L1 and L2 claim} together with the fact that for each 
$l \in L_1 \cup L_2$, $F_k (n)$ contains a cycle of length $l$ with exactly $k$ chord edges completes the proof of Lemma 
\ref{Lemma 1}.  
\end{proof}

\bigskip

The final step is a simple argument that adds no more than $k^2$ chord edges to $F_k (n)$ to account for the 
cycle lengths in the set $\{ n - k + 1 , n - k + 2 , \dots , n \}$. 

\begin{lemma}
We can add no more than $k^2$ chord edges to $F_k(n)$ in such a away that there are cycles of length $l \in \{ n-k+1, n-k+2, \dots, n\}$ that contain these chords.

\end{lemma}
\begin{proof}[Proof of Lemma 11]

Given $k$ chord edges, it is know that we can place these chord edges in a way that creates a unique cycle of any length $l \in \{k, k+1, \dots, n \}$ that contains these chords. Note that in the construction of $F_k(n)$, we are always missing the last $k$ cycles, namely $l \in \{ n-k , n-k+1, \dots, n\}$, no matter the size of $n$. This implies that we can add $k$ chord edges to create a cycle of length $l$ for each $l \in \{ n-k , n-k+1, \dots, n\}$. This would add at most $k^2$ chord edges to $F_k(n)$, and would capture all missing cycles in the original $F_k(n)$ construction. Using a very specific and more complicated construction, it is possible to reduce the amount of chord edges needed to capture these cycles from $k^2$ to $k$, however the strength of the result is not affected by the growth of $n$ in either case. 
\end{proof}

Thus, by Lemma 6 and Lemma 11 $F_k(n)$ has no more than $k\lceil n^{1/k} \rceil + k^2$ chords and contains all cycles of length $l \in \{ k, k+1, \dots, n  \}$.

\section{Concluding Remarks}

We believe that our upper bound is asymptotically best possible and that the coefficient of $n^{1/k}$ in the bound $c(n , k) \leq k \lceil n^{1/k} \rceil + k^2$ is correct.

Affif Chaouche et.\ al.\ ask if $c(n , k)$ is monotone in $n$. We believe that $c(n, k)$ is monotone in $n$ for all $n > (k+2)^k$. Establishing monotonicity for these types of problems seems difficult. For example, Griffin \cite{g} has conjectured that the function $m(n)$, defined in the introduction, satisfies $m(n) \leq m(n+1)$ for all $n \geq 3$, but to our knowledge, this is still open.

\section{Acknowledgments}

I want to thank Craig Timmons and Mike Tait for their valuable comments, input, and assistance during the writing of this paper.


\end{document}